\documentclass[reqno,12pt]{amsart}

\usepackage{amsmath}
\usepackage{amsfonts}
\usepackage{amssymb}
\usepackage{eufrak}
\usepackage{amscd}
\usepackage{amsthm}
\usepackage{epsfig}
\usepackage{amstext}
\usepackage[all]{xy}

\theoremstyle{plain}

 \newtheorem{theorem}{Theorem}
 \newtheorem{example}[theorem]{Example}

 \newtheorem{proposition}[theorem]{Proposition}
 \newtheorem{corollary}[theorem]{Corollary}
 \newtheorem{lemma} [theorem] {Lemma}
 \newtheorem{remark}[theorem]{Remark}
 
  \newtheorem{definition}[theorem]{Definition}

\begin{document}

\title{On subspaces of invariant vectors}
\author{Tatiana Shulman}

\begin{abstract}

Let $X_{\pi}$ be the subspace of fixed vectors for a uniformly bounded representation $\pi$ of a group $G$ on a Banach space $X$. We study the problem of the existence  and uniqueness of a subspace  $Y$ that complements $X_{\pi}$ in $X$. Similar questions for $G$-invariant complement to $X_{\pi}$ are considered.  We prove that every non-amenable discrete group $G$ has a representation with non-complemented $X_{\pi}$ and find some conditions that provide an $G$-invariant complement. A special attention is given to representations on $C(K)$ that arise from an action of $G$ on a metric compact $K$.

\end{abstract}

\address{INSTITUTE OF MATHEMATICS OF THE POLISH ACADEMY OF SCIENCES,
POLAND}

\email{tshulman@impan.pl}



\maketitle

\section*{Introduction}

The subspaces of vectors which are invariant under group representations recently got  new attention because of their use in the Banach space version of Kazhdan's property (T) (see \cite{Monod}, \cite{Piotr}). In the Hilbert space case, arguments used for studying property (T) rely heavily on the existence of orthogonal complements of subspaces (of invariant vectors). In Banach space setting the lack of orthogonality immediately causes difficulties. It is not even clear if the subspace of invariant vectors is always complemented, as mentioned in \cite{Piotr}. However if a Banach space is superreflexive, then for any uniformly bounded representation, the subspace of invariant vectors is complemented (\cite{Monod}). Moreover there is a complement  which is invariant under the representation. Namely it is proved in \cite{Monod} that if $\pi$ is a uniformly bounded representation of a group $G$ on a superreflexive space $X$, then $X$ decomposes  into the sum $$X = X_{\pi} \oplus Ann (X^*_{\bar\pi}),$$ where $X_{\pi}$ is the subspace of invariant vectors, $X^*_{\bar\pi}$ is the subspace of vectors invariant under the dual representation $\bar\pi$ and  $Ann (X^*_{\bar\pi})$ is its preannihilator. It is easy to check that $Ann (X^*_{\bar\pi})$ is $G$-invariant.

In this note we show that indeed the subspace of invariant vectors need not be complemented. Moreover, in section 1 we prove that each non-amenable group admits an isometric representation such that the subspace of invariant vectors is not complemented (Theorem \ref{notcomplemented}).

In section 2 we study the decomposition $X = X_{\pi} \oplus Ann (X^*_{\bar\pi})$,  and more generally, the question of the existence and uniqueness of an invariant complement to $X_{\pi}$ .
 By simple use of fixed point theorems we prove that the decomposition $X = X_{\pi} \oplus Ann (X^*_{\bar\pi})$ holds in some important cases, for instance for   representations of compact groups (Theorem \ref{compactgroup}) and for  uniformly bounded representations on reflexive Banach spaces (Theorem \ref{refl}) which generalizes [\cite{Monod}, Prop. 2.6] since superreflexive spaces are reflexive. In these cases if the representation is isometric, then the corresponding projection onto $X_{\pi}$  has norm 1. Though in general $G$-invariant complement need not be unique (Example \ref{example}), in the cases above it is unique. We show also that for any uniformly bounded representation of an amenable group the subspaces $X_{\pi}$ and $Ann (X^*_{\bar\pi})$ have trivial intersection (Theorem \ref{amen}). For non-amenable groups it is not true in general (\cite{Monod}, ex. 2.29).

In sections 3 and 4 we focus on representations coming from group actions on compact metric spaces. Though for such representations the decomposition
$X = X_{\pi} \oplus Ann (X^*_{\bar\pi})$ need not hold in general, it is shown that it holds if the action is nice, namely Lyapunov stable (Theorem \ref{LyapunovActions}). Lyapunov stable actions were introduced in \cite{Manuilov}. It was shown in \cite{Manuilov} that if the action is Lyapunov stable, then there is a conditional expectation on the subspace (actually, subalgebra) of invariant functions.
In Theorem \ref{LyapunovActions} we give a new proof of that. Moreover we construct a conditional expectation commuting with the representation and show that such an expectation is unique. Along the way we give a short proof of an assertion in \cite{Manuilov} on the uniqueness of invariant measures.
 In section 4 we introduce lower semi-continuous actions, which is a wider class of actions than Lyapunov stable ones.  We show that for lower semi-continuous actions the subspace of invariant functions is complemented.

\bigskip

\noindent {\bf Acknowledgements.} The author would like to thank Piotr Nowak for bringing the topic to her attention  and for helpful discussions.

This research was partially supported
by the Polish National Science Centre grant under the contract number DEC-
2012/06/A/ST1/00256.

\section{The subspace of invariant vectors need not be complemented}

Let $G$ be a topological group and $X$ be a Banach space.  By  a representation of $G$ on $X$ we will mean a strongly continuous homomorphism from $G$ into the group $B(X)$ of bounded invertible operators on $X$. A representation $\pi$ is isometric if $\pi(g)$ is an isometry for each $g\in G$ and $\pi$ is uniformly bounded if $\sup_{g\in G} \|\pi(g)\| < \infty$.

Below for any family $S\subseteq B(H)$ of operators on a Hilbert space by $S'$ we denote its commutant, that is $S' = \{T \in B(H) \;|\; TA = AT, \;\text{for any} \;A\in S \}$.
Recall that the group von Neumann algebra  $L(G)$ of a group $G$ is $$L(G) = \{\rho(g)\;|\;g\in G\}'',$$ where $\rho$ is the left regular representation  of $G$.

\begin{theorem}\label{notcomplemented} Any discrete non-amenable group admits an isometric representation such that the set of invariant vectors is not complemented.
\end{theorem}
\begin{proof} Let $G$ be a non-amenable group, $H = l_2(G)$ and $\rho: G \to B(H)$ --- the left regular representation. Define a representation $\tilde\rho:G \to B(H\otimes H) $ by $$\tilde \rho(g) = \rho(g) \otimes Id.$$ Let $X = B(H\otimes H)$. Define a representation $\pi: G \to B(X)$ by $$\pi(g)x = \tilde \rho(g)x\tilde \rho(g)^{-1},$$
$x\in X$. Let $N$ be the set of $\pi$-invariant vectors. Then \begin{multline}N = \{x\in X \;|\;  \tilde \rho(g)x \tilde \rho(g)^{-1} = x, \forall g\in G\} = \\ \{\tilde \rho(g) \;|\; g\in G\}' = \left\{\rho(g)\otimes Id \;|\; g\in G\right\}' = \\ \left\{ T \otimes Id \;|\; T\in L(G)\right\}' = \left(L(G)\otimes  Id\right)'.\end{multline}
 As is well known a von Neumann algebra is injective if and only if  its commutant is injective (\cite{Connes}).  Since $G$ in non-amenable, $L(G)$ is non-injective (\cite{BrownOzawa}).  Since injectivity is preserved by $\ast$-isomorphisms, $L(G)\otimes Id$ is non-injective either and we conclude that $N$ is non-injective.
  Since $$\left(L(G)\otimes  Id\right)' = L(G)'\otimes B(H), $$ we have $$M_2(N) \cong N.$$ By [\cite{HaagerupPisier}, lemma 4.6], $N$ is not complemented in $X$.

\end{proof}

\bigskip

{\bf Question}: {\it Does there exist a group which admits a uniformly bounded representation on a separable Banach space such that the set of invariant vectors is not complemented?}

\medskip

 {\bf Question}: {\it Does there exist an amenable group which admits a uniformly bounded representation such that the set of invariant vectors is not complemented?}

\section{On the decomposition $X = X_{\pi} \oplus Ann (X^*_{\bar\pi})$}

For a representation $\pi$ one can define its adjoint representation $\bar\pi$ of $G$ on the dual space $X^*$ by $$(\bar\pi(g) f) (x) = f(\pi(g^{-1})x),$$ $x\in X, f\in X^*$.

For a subspace $Y\subseteq X$ by $Ann \;Y$ we will denote its annihilator in $X^*$, that is $$Ann \;Y = \{f\in X^*\;|\; f(x)=0 \;\text{for each}\; x\in Y\}.$$
For a subspace $Y\subseteq X^*$ its preannihilator in $X$ we will denote also by $Ann \;Y$  $$Ann \;Y = \{x\in X\;|\; f(x)=0 \;\text{for each}\; f\in Y\},$$ since it always will be clear from the context what we mean.

Let $X_{\pi}$ be the subspace of $\pi$-invariant vectors  $$X_{\pi} = \{ x\in X \;|\; \pi(g)x= x \;\text{for all} \; g\in G\}.$$

\begin{lemma}\label{1} If $X^*_{\bar\pi} + Ann (X^{**}_{\overline {\overline \pi}})$ is $\ast$-weakly dense in $X^*$, then $X_{\pi}\cap Ann (X^*_{\bar\pi}) ~= ~\{0\}$.
\end{lemma}
\begin{proof} It is easy to see that $X_{\pi}\cap Ann (X^*_{\bar\pi})$ annihilates the subspace  $X^*_{\bar\pi} + Ann (X^{**}_{\overline {\overline \pi}})$.
\end{proof}

\begin{lemma}\label{2} Let $\pi$ be a representation of a group $G$ on a Banach space $X$.

(i) Suppose that for any $x\in X$, the closed convex hull $E(x)$ of its orbit $O(x)$ contains an invariant vector. Then $$X = X_{\pi} + Ann (X^*_{\bar\pi}).$$

(ii) Suppose that for any $f\in X^*$, the $\ast$-weakly closed convex hull $E_w(f)$ of $O(f)$ contains an invariant vector. Then $$X^* = X^*_{\bar\pi} + Ann (X_{\pi}).$$
\end{lemma}\label{L4}
\begin{proof} (i) By the assumption, for $x\in X$, there is an invariant vector $x_0\in E(x)$. Let $f\in X^*_{\bar \pi}$. Then $f$ is constant on $O(x)$ and hence on $E(x)$. Hence $f(x-x_0)=0$. Thus $x-x_0\in Ann (X^*_{\bar\pi})$ and
 $X = X_{\pi} + Ann (X^*_{\bar\pi}).$

 (ii) For $f\in X^*$, let $f_0$ be an invariant functional in $E_w(f)$. Let $x\in X_{\pi}$, then $f(x) = f_0(x)$, and therefore $f-f_0 \in Ann(X_{\pi})$, $X^* = X^*_{\bar\pi} + Ann (X_{\pi})$.
\end{proof}

\begin{theorem}\label{refl} For any uniformly bounded representation $\pi$ on a reflexive Banach space $X$, $$X = X_{\pi} \oplus Ann (X^*_{\bar\pi}).$$
The corresponding projection onto $X_{\pi}$ has norm less or equal to $\sup_{g\in G} \|\pi(g)\|$.
\end{theorem}
\begin{proof}  Define an equivalent norm on $X$ by $$\|x\|' = \sup_{g\in G} \|\pi(g)x\|.$$ With respect to this norm $\pi$ is isometric.
So we can assume from the beginning that $\pi$ is an isometric representation on  a reflexive space $X$.

\noindent Let $x\in X$. Let $O(x)$ and $E(x)$ be as in Lemma \ref 2. Since $X$ is reflexive,  $E(x)$ is a convex weakly compact invariant subset. By Ryll-Nardzewski theorem in $E(x)$ there is an invariant vector. By Lemma  \ref{2}, $X = X_{\pi} + Ann (X^*_{\bar\pi}).$ Since $X^*$ is also reflexive, we in the same way obtain $X^* = X^*_{\bar\pi} + Ann (X^{**}_{\overline {\overline \pi}})$. By Lemma \ref 1, $X = X_{\pi} \oplus Ann (X^*_{\bar\pi}).$

 \noindent Let $P$ be the corresponding projection onto $X_{\pi}$.  By the proof of Lemma \ref {2}, for any $x\in X$, $Px$ belongs to the closed convex hull of the orbit of $x$. Hence $\|Px\|\le sup_{g\in G}\|\pi(g)x\|$,  for any $x\in X$,  and $\|P\| \le \sup_{g\in G} \|\pi(g)\|$.
\end{proof}

\bigskip

\begin{theorem}\label{compactgroup} For any representation $\pi$ of a compact group on a Banach space $X$, $$X = X_{\pi} \oplus Ann (X^*_{\bar\pi}).$$ If $\pi$ is isometric, then the corresponding projection onto $X_{\pi}$ has norm 1.
\end{theorem}
\begin{proof} Let $x\in X$. Let $O(x)$ and $E(x)$ be as in Lemma \ref 2. Since the group is compact, $O(x)$ is compact. Hence $E(x)$ is a convex compact invariant subset. By Kakutani's fixed point theorem in $E(x)$ there is an invariant vector. By Lemma \ref{2} $X = X_{\pi} + Ann (X^*_{\bar\pi}).$ The same applies to $\overline \pi$, so we also get $X^* = X^*_{\bar\pi} + Ann (X^{**}_{\overline {\overline \pi}})$. By Lemma \ref{1},
$$X = X_{\pi} \oplus Ann (X^*_{\bar\pi}).$$

\noindent The same argument as in Theorem \ref{refl} shows that the projection onto $X_{\pi}$ has norm 1.
\end{proof}

\bigskip

\begin{theorem}\label{amen} If $\pi$ is a uniformly bounded representation of an amenable group on a Banach space $X$ then

(i) $X^* = X^*_{\bar\pi} + Ann(X_{\pi}),$

(ii) $X_{\pi}\cap  Ann (X^*_{\bar\pi}) =  \{0\}.$
\end{theorem}
\begin{proof} (i) For each $f\in X^*$, the $\ast$-weakly closed convex span $E_w(f)$ of $O(f)$ is $\ast$-weakly compact. Since all the operators $\bar\pi(g), \;g\in G$, are $\ast$-weakly continuous, it follows from amenability that $E_w(f)$ contains a fixed point of $\bar\pi$. Now by Lemma \ref{L4} we conclude that $X^* = X^*_{\bar\pi} + Ann(X_{\pi}).$

(ii) It is easy to see that $X^*_{\bar\pi} + Ann(X_{\pi})$ annihilates $X_{\pi}\cap  Ann (X^*_{\bar\pi})$.
\end{proof}

Note that the decomposition in (i) is not necessarily direct. For example if $X = l^1(\mathbb{Z})$, $G = \mathbb{Z}$ and $\pi(n)f(k ) = f(n+k)$ then $X_{\pi} = 0$, $Ann(X_{\pi}) = X^*$, $X^*_{\bar\pi}$ is the space of constant sequences. Similarly in this example $X\neq X_{\pi} +  Ann (X^*_{\bar\pi})$.







\medskip

\begin{proposition}\label{unique} If $X_{\pi}$ has a  $\pi(G)$-invariant complement $Y$, then $Y \supseteq Ann (X^*_{\bar\pi}).$ In particular, in  Theorems \ref{refl} and \ref{compactgroup} the space $X_{\pi}$ has unique $\pi(G)$-invariant complement.
\end{proposition}
\begin{proof} Let $P$ be the projection onto $X_{\pi}$ parallel to $Y$. Since $Y$ is $\pi(G)$-invariant,  $[P, \pi(g)] = 0$, for all $g\in G$. Hence for any $x\in X$
$$P(\pi(g)x-x)= \pi(g)Px-Px =0.$$ Thus $\pi(g)x-x\in Y$.
Let $f\in X^*$ and $f|_Y=0$. Then $f(\pi(g)x-x)=0$, for all $g\in G$, $x\in X$, that is $f\in X^*_{\bar\pi}$. Hence $Y^{\bot} \subseteq X^*_{\bar\pi}$, whence $Y \supseteq Ann (X^*_{\bar\pi}).$
\end{proof}

\medskip

In general a $\pi(G)$-invariant complement need not be unique as shows Example \ref{example} below.

\section{Lyapunov stable actions}
Let $K$ be a metric compact space and let group $G$ act continuously on $K$. Define a representation $\pi$ of $G$ on $C(K)$ by $$\pi(g)\phi (x) = \phi (g^{-1}x).$$

\noindent The following example shows that the decomposition $$C(K)= C(K)_{\pi} \oplus Ann (C(K)^*_{\bar\pi})$$ does not hold in general, even when the group is abelian. It also shows that $\pi(G)$-invariant complement need not be unique.

\begin{example}\label{example} \rm  Let $K = [0, 1]$ and let $G$ be the group $\mathbb Q_+ = \{\frac{m}{n}\;|\; m, n\in \mathbb N\}$ with the usual multiplication and the discrete topology. Define a continuous action $\alpha$ of $G$ on $K$ as $$\alpha\left(\frac{m}{n}\right)x = x^{\frac{m}{n}}.$$ Define a representation $\pi$ of $G$ on $C(K)$ by $$\pi(g)\phi (x) = \phi (\alpha(g^{-1})x).$$  Let us show that $C(K)_{\pi}$ is the subspace of constant functions. For each $x\in [0, 1)$,  $0\in \overline{O(x)}$. Hence for $\phi\in C(K)_{\pi}$ and each $x\in [0, 1)$, $\phi(x)=\phi(0)$. Thus $\phi = const$. We will show now that $Ann ((C(K))^*_{\bar\pi})\subseteq C_0(0,1)$. Define  $h_i\in C(K)^*$, $i=1, 2$, as $h_1(\phi) = \phi(0)$, $h_2(\phi) = \phi(1)$, for any $\phi\in C(K)$. It is easy to see that $h_i$, $i=1, 2$, are constant on orbits of functions in $C(K)$ and hence $h_i\in (C(K))^*_{\bar\pi}$, $i=1, 2$. Hence $$ Ann ((C(K))^*_{\bar\pi}) \subseteq Ann (h_1) \cap Ann (h_2) = C_0(0, 1] \cap C_0[0, 1) = C_0(0, 1).$$ Thus $Ann ((C(K))^*_{\bar\pi})$ does not complement $C(K)_{\pi}$. However $C(K)_{\pi}$ has $\pi(G)$-invariant complements $C_0(0, 1] = \{\phi \in C(K) \; | \; \phi(0) = 0\}$ and $C_0[0, 1) = \{\phi \in C(K) \; | \; \phi(1) = 0\}$ (and many others).
\end{example}

\bigskip

However we will show that if an action is nice enough (namely, Lyapunov stable)  then the decomposition $$C(K)= C(K)_{\pi} \oplus Ann (C(K)^*_{\bar\pi})$$ holds.

\begin{definition} An action of $G$ on $K$ is {\it Lyapunov stable} if for any $\epsilon >0$ there is $\delta = \delta(\varepsilon) >0$ such that $d(gx, gy) < \varepsilon$, for all $g\in G$,  whenever $d(x, y) <\delta$.
\end{definition}

\begin{remark} \rm The original definition in \cite{Manuilov} was different: for any $x\in K$ and $\epsilon >0$ there must exist $\delta = \delta(x,\varepsilon) >0$ such that $d(gx, gy) < \varepsilon$, for all $g\in G$,  whenever $d(x, y) <\delta$. But a standard compactness argument shows that for compact $K$ the definitions coincide. Indeed for each $x\in K$, let $U_x = \{y\in K: d(y,x) < \delta(x,\varepsilon/2)/2\}$ and choose a finite subcovering $U_{x_1},...,U_{x_n}$. Let $\delta = \min_i  \delta(x_i,\varepsilon/2)/2$, $i =1,...,n$. If $d(x,y) \le \delta$, one can find $i$ with $x\in U_{x_i}$, then $d(y,x_i) \le \delta+\delta(x_i,\varepsilon/2)/2 \le \delta(x_i,\varepsilon/2)$. It follows that $d(gx,gx_i) \le \varepsilon/2$ and $d(gy,gx_i) \le \varepsilon/2$ whence $d(gx,gy) \le \varepsilon$.
\end{remark}

Below by $\pi$ we always mean the representation coming from some group action on a compact.

\begin{lemma}\label{compactorbits} Let an action of $G$ on $K$ be Lyapunov stable and $\pi$ be as above. Then for any $\phi \in C(K)$, its orbit $O(\phi)$ is precompact.
\end{lemma}
\begin{proof} It is easy to see that Lyapunov stability implies that for any $\phi\in C(K)$, $O(\phi)$ is an equicontinuous family of functions.
Since $O(\phi)$ is bounded, it is precompact by Arzela –- Ascoli theorem.
\end{proof}

\begin{lemma}\label{ForCor} Let an action of $G$ on $K$ be Lyapunov stable and $\pi$ be as above. Then $$C(K) = C(K)_{\pi} + Ann (C(K)^*_{\bar\pi}).$$
\end{lemma}
\begin{proof} It follows from Lemma \ref{compactorbits}, that for any $\phi \in C(K)$, the closed convex hull $E(\phi)$ of the orbit $O(\phi)$ is compact. By  Kakutani's fixed point theorem $E(\phi)$ contains a common fixed point for all $\pi(g)$, that is an invariant vector. By Lemma \ref{2}, \begin{equation}\label{eq1} C(K) = C(K)_{\pi} + Ann (C(K)^*_{\bar\pi}).\end{equation}
\end{proof}

Now we obtain a short proof of an assertion in \cite{Manuilov} on the uniqueness of invariant measures.

\begin{corollary}(\cite{Manuilov}, Lemma 6.1) Suppose a group $G$ acts on a compact metric space $K$ in such a way that the orbit of an element $a\in K$ is dense in $K$. If the action is Lyapunov stable, then $K$ carries not more than one invariant regular measure.
\end{corollary}
\begin{proof} By Lemma \ref{ForCor},  $C(K) = C(K)_{\pi} + Ann (C(K)^*_{\bar\pi}).$ Since the orbit of $a\in K$ is dense in $K$, the only invariant functions are constants, so $C(K)_{\pi} ~=~ \mathbb C$.  Hence $codim Ann (C(K)^*_{\bar\pi}) \le 1$ and  $\dim C(K)^*_{\bar\pi} \le 1$.
The latter exactly means that there is not more than one invariant regular measure on $K$, because regular measures are in one-to one correspondence with points of $C(K)^*$ by Riesz theorem.
\end{proof}

\begin{lemma}\label{equicontinuous} Let $\pi$ be a representation  of a group $G$ on a Banach space $X$ such that all orbits are precompact. Let $E\subset X^*$ be a  bounded $\bar\pi(G)$-invariant subset. Then the group of maps $\{\bar\pi(g)|_{E}\}_{g\in G}$ is equicontinuous with respect to the relative $\ast$-weak topology on $E$.
\end{lemma}
\begin{proof} Let $\epsilon >0$ and $x_1, \ldots, x_N\in X$. We need to find $\delta>0$ and $y_1, \ldots, y_m\in X$ such that $$|(\bar\pi(g)(h_1-h_2))(x_i)|<\epsilon, \; g\in G, \; i=1, \ldots, N$$ whenever $h_1, h_2\in E$ and $|h_1(y_j) - h_2(y_j)| < \delta$, $j=1, \ldots, m$.
Let $M = \sup_{h\in E} \|h\|$. As $\{y_j\}$ we take a $\frac{\epsilon}{6M}$-net in $\bigcup_{i=1}^N \overline{O(x_i)}$ which exists since all orbits are precompact. Let $\delta = \epsilon/3$. Then for any $h_1, h_2\in E$ such that $|h_1(y_j) - h_2(y_j)| < \delta$, $j=1, \ldots, m$, and for any $g\in G$ we have
$$|(\bar\pi(g)(h_1-h_2))(x_i)| = |(h_1-h_2)(\pi(g^{-1})x_i)| \le |(h_1-h_2)(\pi(g^{-1})x_i - y_k)| + |(h_1-h_2)(y)| $$ (where $k=k(i, g)$ is chosen in such  a way that $\|\pi(g^{-1})x_i - y_k\| \le \frac{\epsilon}{6M}$) $$\le \|h_1-h_2\|\frac{\epsilon}{6M} + \epsilon/3 < \epsilon/3 + \epsilon/3 \le \epsilon.$$
\end{proof}

\begin{theorem}\label{LyapunovActions} Let an action of $G$ on $K$ be Lyapunov stable and $\pi$ be as above. Then $$C(K) = C(K)_{\pi} \oplus Ann (C(K)^*_{\bar\pi}).$$ The corresponding projection onto $C(K)_{\pi}$ is a conditional expectation.
\end{theorem}
\begin{proof} Let $f\in C(K)^*$ and $E(f)$ be the closed convex hull of $O(f) = \{\bar\pi(g)f\}_{g\in G}$. Then $E(f)$ is a convex $\ast$-weakly compact $\bar\pi(G)$-invariant subset.  By Lemma \ref{equicontinuous} and Kakutani's fixed point theorem, in $E(f)$ there is a $\bar\pi(G)$-invariant vector.
By Lemma \ref{2} \begin{equation}\label{eq2}C(K)^* = C(K)^*_{\bar\pi} + Ann (C(K)^{**}_{\overline {\overline \pi}}). \end{equation} The decomposition $C(K) = C(K)_{\pi} \oplus Ann (C(K)^*_{\bar\pi})$ follows now from Lemma \ref{1}, Lemma \ref{ForCor} and (\ref{eq2}). The same argument as in Theorem \ref{refl} shows that the projection onto $X_{\pi}$ has norm 1. Since $C(K)$ and $C(K)_{\pi}$ are $C^*$-algebras, by \cite{Tomiyama} it is a conditional expectation.
\end{proof}

\section{Lower semi-continuous actions}

Now we will show that for more general actions than Lyapunov stable, namely for lower semi-continuous actions, the subspace $C(K)_{\pi} $ is complemented.

\medskip

Let $X$ be a compact metric space.
Let $M$ be a partition of $X$
into closed subsets. For $x\in  X$ let $M(x)$ denote the member of $M$
which contains $x$. Corresponding to the standard definitions  $M$ is called {\it lower semi-continuous} if
$\{x ~\in X: ~M(x) \bigcap U \neq \emptyset \}$
is an open set in $X$ for every open set $U$ in $X$.


\medskip

If $P \subseteq C(X)$  is a subspace, then the
{\it $P$-partition} of $X$ is the partition associated with the following equivalence
relation R. A couple $(x, y)\in R$ if and only if $p(x) = p(y)$
for every $p \in P$. Now let
 \begin{multline} K(P) = \bigcup \{K \subseteq X: K \text{\;is a
member of the P-partition of X},  \\ \text{\;and K contains more than one point
of}\; X\}\end{multline}
According to \cite{Wulbert},  $P$ has a {\it lower semi-continuous quotient} if
the restriction of the $P$-partition to $\overline{K(P)}$ is lower semi-continuous.

\begin{definition} We will say that an action is lower semi-continuous if the corresponding $C(K)_{\pi} $-partition is lower semi-continuous.
\end{definition}

\medskip

We don't know any example of a not lower semi-continuous action. An easy example of a lower semi-continuous action is the action from Example \ref{example}. Other examples are Lyapunov stable actions as we show below.

\begin{lemma}\label{open} For any continuous  action and for any open $U\subseteq K$, the set $\{x\;|\;  O(x) \cap U \neq \emptyset\}$ is open.
\end{lemma}
\begin{proof} Let $E = \{x\;|\;  \overline{O(x)} \cap U \neq \emptyset\}$. Since $U$ is open,  $x\in E$ implies that there is $g\in G$ and $\epsilon >0$ such that the ball $B_{\epsilon}(gx)$ is inside $U$. There is $\delta$ such that $d(gx, gy) \le \epsilon$ whenever $d(x, y)<\delta$. Thus for any $y\in B_{\delta}(x)$ we have $gy \in B_{\epsilon}(gx))\subset U$. Hence $y\in E$. Thus $E$ is open.
\end{proof}

\medskip

We are going to use the following result from \cite{Manuilov}.
\begin{theorem}\label{FMT}(\cite{Manuilov}, Lemma 3.1) For a Lyapunov stable action, any two orbits are either separated from each other, or have the same closure. The quotient space of closures of orbits is Hausdorff.
\end{theorem}

\medskip

The following corollary shows that for Lyapunov stable actions the partition into closed orbits and the $C(K)_{\pi}$-partition are the same.

\begin{corollary}\label{EqualPartitions} Suppose we have a Lyapunov stable action and let $R$ be the equivalence relation  defining $C(K)_{\pi}$-decomposition of $K$.
Then $(x, y)\in R$ if and only if $ \overline{O(x)} =  \overline{O(y)}$.
\end{corollary}
\begin{proof} Since functions in $C(K)_{\pi}$ are those which are constant on orbits,  the "if" part follows.

\noindent To prove the "only if" part, assume that $ \overline{O(x)} \neq  \overline{O(y)}$. Then by Urysohn's lemma and Theorem \ref{FMT}, there is a continuous function $\psi$ on $K/s$ such that $\psi(\overline{O(x)}) \neq \psi(\overline{O(y)}).$  Define $\phi\in C(K)$ by $\phi(x) = \psi(\overline{O(x)}).$ Then $\phi\in C(K)_{\pi}$ and $\phi(x) \neq \phi(y)$, hence $(x, y)\notin R$.
\end{proof}

\begin{theorem} Lyapunov stable actions are lower semi-continuous.
\end{theorem}
\begin{proof} For any open $U\subseteq K$ we need to check that the set $$E = \{x\in K\;|\; R(x, u)=0, \text{for some}\; u\in U\}$$ is open. By Corollary \ref{EqualPartitions} and Theorem \ref{FMT} $$ E = \{x\;|\; \overline{O(x)} = \overline{O(u)}, \text{for some}\; u\in U\} = \{x\;|\; \overline{O(x)} \cap U \neq \emptyset\} = \{x\;|\; O(x) \cap U \neq \emptyset\}$$ which is open by Lemma \ref{open}.
\end{proof}

\medskip

Now we will show that for lower semi-continuous actions, the subspace $C(K)_{\pi} $ is always complemented.

\begin{lemma}\label{quotient} Let  $P \subseteq C(X)$  be a subspace such that $P$-partition  of $X$ is lower semi-continuous.  Then $P$ has lower semi-continuous quotient.
\end{lemma}
\begin{proof}  It suffices to show that a subpartition of a lower semi-continuous partition is lower semi-continuous.
Let a partition $M$ be lower semi-continuous  and $M_0$ be its subpartition. Let $K_0$ be the closure of the union of members of $M_0$.
Suppose that $U$ is open in $K_0$. We need to show that $\{x\in K_0|\; M_0(x)\bigcap U \neq \emptyset\}$ is open in $K_0$. Since $U \bigcup \{K \setminus K_0\}$ is open in $K$ (because its complement is $K_0\setminus U$) and $M$ is lower semi-continuous, the set $$\{x\in K|\; M(x)\bigcap (U\bigcup \{K\setminus K_0\})\neq \emptyset\}$$ is open, whence $$\{x\in K_0|\; M_0(x) \bigcap U  \neq \emptyset\} = \{x\in K|\; M(x) \bigcap (U \bigcup \{K\setminus K_0\})\neq \emptyset\}\bigcap K_0$$ is open in $K_0$.
\end{proof}

\begin{proposition}\label{compl} Suppose  $G$-action  on $K$ is lower semi-continuous. Then $C(K)_{\pi} $ is complemented.
\end{proposition}
\begin{proof} Obviously $C(K)_{\pi} $ is a $C^*$-subalgebra of $C(K)$ and hence is isomorphic to $C(Z)$, for some Hausdorff space $Z$. The statement follows now from Lemma \ref{quotient} and [Th. 4, \cite{Wulbert}].
\end{proof}.

\end{document}